\documentclass[a4paper, oneside]{article}
\usepackage{enumitem}
\usepackage{lmodern}
\usepackage{amssymb,amsthm,amsmath}
\usepackage[utf8]{inputenc}
\usepackage{verbatim}
\usepackage{tikz}
\usepackage{authblk}
\usepackage{cite}
\usepackage{hyperref}

\usetikzlibrary{math} 
\usetikzlibrary{patterns} 
\usetikzlibrary{arrows}
\usetikzlibrary{decorations.pathreplacing}
\usetikzlibrary{decorations.pathmorphing}
\usetikzlibrary{automata}
\usetikzlibrary{positioning}

\theoremstyle{definition}
\newtheorem{definition}{Definition}[section]
\newtheorem{lemma}[definition]{Lemma}

\newtheorem{theorem}[definition]{Theorem}
\newtheorem{example}[definition]{Example}

\newtheorem{remark}[definition]{Remark}
\newtheorem{problem}[definition]{Problem}

\newcommand{\N}{\mathbb{N}}
\newcommand{\Zpos}{{\mathbb{Z}_+}}

\newcommand{\R}{\mathbb{R}}
\newcommand{\Z}{\mathbb{Z}}

\newcommand{\abs}[1]{{\left\vert #1 \right\vert}}

\newcommand{\digs}{\Sigma} 
\newcommand{\alp}{\digs} 
\newcommand{\Mul}{\Pi} 
\newcommand{\mul}{g} 

\DeclareMathOperator{\fractional}{frac} 
\DeclareMathOperator{\config}{config} 
\DeclareMathOperator{\real}{real} 
\DeclareMathOperator{\tr}{Tr} 
\DeclareMathOperator{\num}{\mathcal{N}} 
\DeclareMathOperator{\lbound}{\ell} 

\begin{document}

\title{A natural class of cellular automata containing fractional multiplication automata, Rule~30, and others}

\author{Johan Kopra}

\affil{Department of Mathematics and Statistics, \\FI-20014 University of Turku, Finland}
\affil{jtjkop@utu.fi}

\date{}

\maketitle

\setcounter{page}{1}

\begin{abstract}
\noindent We define the class of rapidly left expansive cellular automata, which contains fractional multiplication automata, Wolfram's Rule 30, and many others. The definition has been shaped by a proposition of Jen on aperiodicity of columns in space-time diagrams of certain cellular automata, which generalizes to this new class. We also present results that originate from the theory of distribution modulo~1. 
\end{abstract}

\providecommand{\keywords}[1]{\textbf{Keywords:} #1}
\noindent\keywords{cellular automata, symbolic dynamics, distribution modulo 1, fractional multiplication automata, Wolfram's Rule~30}

\section{Introduction}

Distribution of fractional parts (i.e. distribution modulo 1) of sequences of the form $(\xi(p/q)^i)_{i\in\N}$ for $\xi>0$ and integers $p>q>1$ is a mysterious topic as demonstrated e.g. in Chapter~3 of the book~\cite{Bug12}. For example, in the case $p/q=3/2$, it is not known whether $\xi>0$ can be chosen so that fractional parts in the whole sequence remain less than $1/2$~\cite{Mah68}. One way to approach this topic is via symbolic dynamical systems called cellular automata (CA) (for a survey on CA, see~\cite{Kari05}). Indeed, multiplication by a fraction $p/q$ in base $pq$ can be implemented by a cellular automaton~\cite{Kari12a} that we denote by $\Mul_{p/q,pq}$, and then results on the distribution of fractional parts can be proven by analyzing the symbolic dynamics of this CA as in~\cite{KK17}.

Due to this connection, in this paper instead of using CA as a tool for proving results on the distribution of fractional parts, we do the opposite and use results on the distribution of fractional parts as an inspiration for new results on cellular automata. More concretely, we choose as a starting point the results saying that any fractional part repeats in the sequence $(\xi(p/q)^i)_{i\in\N}$ only finitely many times~\cite{Dub08} and that the fractional parts of this sequence have infinitely many limit points~\cite{Pis46}. Using the CA $\Mul_{p/q,pq}$ these results can be reformulated in a symbolic dynamical form. Furthermore, it is possible to give alternative, purely symbolic dynamical proofs of these results, and most importantly, these proofs are not specific to the CA $\Mul_{p/q,pq}$, so these results can be generalized to a wide class of CA that we call rapidly left expansive cellular automata. In Section~\ref{rapidSect} we present the definition of rapidly left expansive CA and prove the essential Theorem~\ref{aperThm} that has shaped this definition, a generalization of a result of Jen~\cite{Jen90} on aperiodicity of columns in space-time diagrams of certain CA. It is notable that this class includes Wolfram's Rule~30, a cellular automaton which is notoriously resistant to proofs of nontrivial results and has recently inspired Stephen Wolfram to offer prizes for the solution of certain problems concerning it~\cite{Wol19}. Our results on rapidly left expansive CA in Section~\ref{resSect} can be applied to provide new nontrivial information on the asymptotic behavior of Rule~30.

\section{Preliminaries}

We denote the set of positive integers by $\Zpos$ and define the set of natural numbers by $\N=\Zpos\cup\{0\}$. Whenever $A$ and $B$ are sets, $B^A$ denotes the collection of functions from $A$ to $B$.

We call a nonempty finite set $\alp$ of symbols an alphabet. We will assume without loss of generality that $\alp$ is equal to $\digs_n=\{0,1,\dots,n-1\}$ for some $n\in\Zpos$, so in particular $\alp$ always contains $0$. For a set $A$ and an alphabet $\alp$, we typically denote the value of a function $f\in\alp^A$ at $a\in A$ by $f[a]$ instead of $f(a)$. We denote by $0^A\in\alp^A$ the special function satisfying $0^A[a]=0$ for all $a\in A$. Bi-infinite sequences over an alphabet $\alp$ are called configurations. A configuration $x$ is formally an element of $\alp^\Z$ and therefore its value at a coordinate $i\in\Z$ is denoted by $x[i]$.

A configuration $x\in \alp^\Z$ (respectively, a sequence $x\in \alp^\N$) is periodic if there is a $p\in\Zpos$ such that $x[i+p]=x[i]$ for all $i\in\Z$ (respectively, $i\in\N$). Then we may also say that $x$ is $p$-periodic. We say that $x\in \alp^\Z$ (respectively, $x\in \alp^\N$) is eventually periodic if there are $p\in\Zpos$ and $c\in\Z$ (respectively, $c\in\N$) such that $x[i+p]=x[i]$ holds for all $i\geq c$. When $x\in\alp^\N$, such an $c\in\N$ is called a preperiod of $x$.

Any finite sequence $v=v[0] v[1]\cdots v[n-1]$, where $n\in\N$ and $v[i]\in \alp$, is a word over $\alp$. We say that the word $v$ occurs in a configuration $x\in\alp$ at position $i$ if $x[i]\cdots x[i+n-1] = v[0]\cdots v[n-1]$. The set of words of length $n$ over $\alp$ is denoted by $\alp^n$.

A sequence $(x_i)_{i\in\N}$ with $x_i\in\alp^\N$ converges to $x\in\alp^\N$ if for every $n\in\N$ there exists an $N\in\N$ such that $x_i[0,n]=x[0,n]$ for all $i\geq N$. We say that $x$ is a limit point of $(x_i)_{i\in\N}$ if some subsequence of $(x_i)_{i\in\N}$ converges to $x$. We briefly mention that these agree with the usual definitions of convergence and limit points when $\alp^\N$ is equipped with the prodiscrete topology. 

\begin{definition}
Let $\alp$ be an alphabet. We say that a map $F:\alp^\Z\to \alp^\Z$ is a \emph{cellular automaton} (with memory $m$ and anticipation $n$ for $m,n\in\N$) if there exists a map $f:\alp^{m+n+1}\to \alp$ such that $F(x)[i]=f(x[i-m],\dots,x[i],\dots,x[i+n])$ for $i\in\Z$. Such a map $f$ is an $(m,n)$ \emph{local rule} of $F$. If we can choose $r\in\N$ so that $r=m=n$, we say that $F$ is a radius-$r$ CA.
\end{definition}
Note also that if $F$ has an $(m,n)$ local rule $f:\alp^{m+n+1}\to \alp$, then $F$ is a radius-$r$ CA for $r=\max\{m,n\}$, with possibly a different local rule $f':\alp^{2r+1}\to \alp$.

One of the simplest cellular automata is the shift map $\sigma:\alp^\Z\to \alp^\Z$ defined by $\sigma(x)[i]=x[i+1]$ for $x\in \alp^\Z$, $i\in\Z$: this clearly has a $(0,1)$ local rule. In the following we present more interesting examples of cellular automata, which will be used to motivate the definition of rapidly left expansive cellular automata in the next section.

\begin{example}[Left permutive CA]
A CA $F:\alp^\Z\to\alp^\Z$ is $(m,n)$ \emph{left permutive} with $m\in\Zpos$ and $n\in\N$ if $F$ has an $(m,n)$ local rule $f:\alp^{m+n+1}\to \alp$ such that for every word $v\in\alp^{m+n}$ the function $f_v:\alp\to\alp$ defined by $f_v(a)=f(av)$ for $a\in\alp$ is a bijection.
\end{example}

\begin{example}[Elementary CA]
Elementary cellular automata (ECA) are CA $F:\digs_2^\Z\to\digs_2^\Z$ with a binary alphabet and with $(1,1)$ local rules. Their study was initiated by Wolfram in the paper~\cite{Wol83}, which also popularized a systematic naming scheme for them. Using that naming scheme, one notable example is the Rule~30 automaton $W_{30}$ with the $(1,1)$ local rule $f_{30}:\digs_2^3\to\digs_2$ defined by
\begin{flalign*}
& f_{30}(000)=0 \quad f_{30}(001)=1 \quad f_{30}(010)=1 \quad f_{30}(011)=1 \\
& f_{30}(100)=1 \quad f_{30}(101)=0 \quad f_{30}(110)=0 \quad f_{30}(111)=0.
\end{flalign*}
Rule~30 is $(1,1)$ left permutive, because both symbols of $\digs_2$ appear in each of the four columns above.
\end{example}

Before the following example, fractional multiplication CA, it is appropriate to define the notion of number-like configurations.

\begin{definition}We say that a configuration $x\in \alp^\Z$, $x\neq 0^\Z$, is \emph{number-like} if there exists an $N\in\Z$ such that $x[i]=0$ for $i<N$ . Then the minimal $N\in\Z$ such that $x[N]\neq 0$ is the \emph{left edge} of $x$, denoted by $\lbound(x)$. The set of all number-like configurations over $\alp$ is denoted by $\num(\alp)$.
\end{definition}

Number-like configurations are analogous to usual representations of positive numbers, where there may be infinitely many digits to the right of the decimal point but always a finite number of digits to the left of the decimal point (and then the representation can be extended to a bi-infinite sequence by adding an infinite sequence of zeroes to the left end).

\begin{example}[Fractional multiplication CA]
Let $n>1$. If $\xi>0$ is a real number and $\xi=\sum_{i=-\infty}^{\infty}{\xi_i n^{i}}$ is the unique base $n$ expansion of $\xi$ such that $\xi_i\neq n-1$ for infinitely many $i<0$, we define $\config_n(\xi)\in \num(\digs_n)$ by
\[\config_n(\xi)[i]=\xi_{-i-1}\]
for all $i\in\Z$. In reverse, for $x\in \num(\digs_n)$ we define
\[\real_n(x)=\sum_{i=-\infty}^{\infty}{x[-i] n^{i-1}}.\]
Clearly $\real_n(\config_n(\xi))=\xi$ and $\config_n(\real_n(x))=x$ for every $\xi>0$ and every $x\in\num(\digs_n)$ such that $x[i]\neq n-1$ for infinitely many $i>0$.

For coprime $p>q>1$ we define a $(0,1)$ local rule $\mul_{p,pq}:\digs_{pq}\times \digs_{pq}\to \digs_{pq}$ for a CA $\Mul_{p,pq}$, which performs multiplication by $p$ in base $pq$ in the sense that $\real_{n}(\Mul_{p,pq}(\config_{n}(\xi)))=p\xi$ for all $\xi>0$. Digits $a,b\in \digs_{pq}$ are represented as $a=a_1q+a_0$ and $b=b_1q+b_0$, where $a_0,b_0\in \digs_q$ and $a_1,b_1\in \digs_p$: such representations always exist and they are unique. Then
\[\mul_{p,pq}(a,b)=\mul_{p,pq}(a_1q+a_0,b_1q+b_0)=a_0p+b_1.\]
The map $\mul_{p,pq}$ encodes the usual algorithm for long multiplication by $p$ in base $pq$ (for more details, see e.g.~\cite{Kari12a,Kop21trace}). It is also possible to define a \emph{fractional multiplication automaton} $\Mul_{p/q,pq}$ that multiplies by $p/q$ in base $pq$ as the composition $\sigma^{-1}\circ \Mul_{p,pq}\circ \Mul_{p,pq}$. This CA has a $(1,1)$ local rule.
\end{example}

\begin{definition}
A space-time diagram $\theta\in\alp^{\Z\times(-\N)}$ (of a configuration $x\in\alp^\Z$ with respect to a CA $F:\alp^\Z\to\alp^\Z$) is defined by $\theta[(i,j)]=F^{-j}(x)[i]$ for $(i,j)\in\Z\times(-\N)$.
\end{definition}
The space-time diagram of $x$ with respect to $F$ is usually depicted by drawing the configurations $x,F(x),F^2(x),\dots$ on consecutive rows as in Figure~\ref{stds}.

\begin{figure}[ht]
\centering
\includegraphics[scale=0.26]{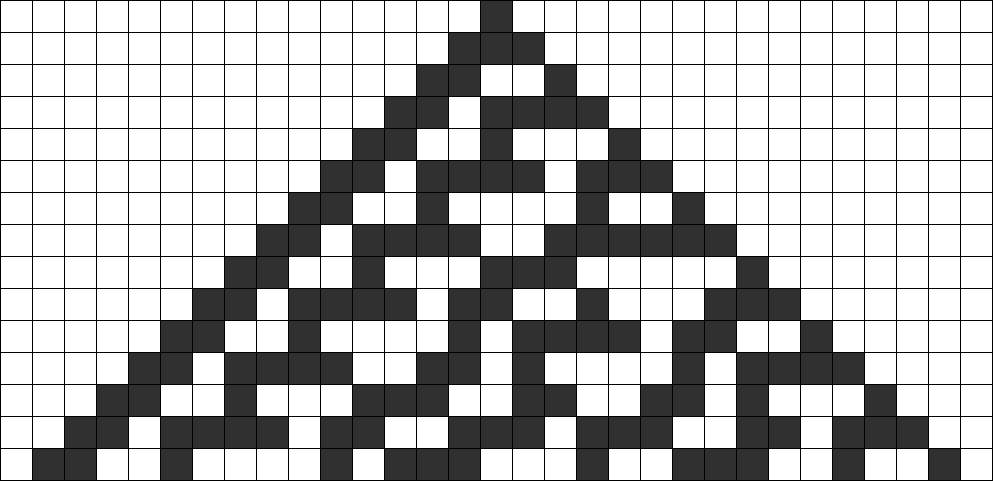}\qquad
\includegraphics[scale=0.26]{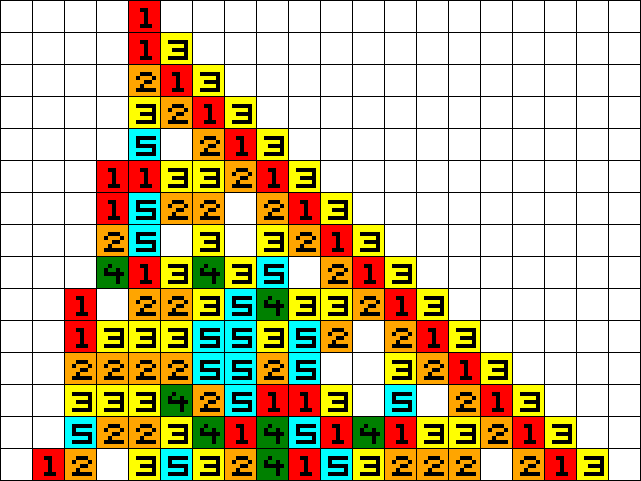}
\caption{The space-time diagrams of the configuration $\cdots 0001000\cdots$ with respect to the CA $W_{30}:\digs_2^\Z\to\digs_2^\Z$ and $\Mul_{3/2,6}:\digs_6^\Z\to\digs_6^\Z$. White and black squares correspond to digits $0$ and $1$ respectively.}
\label{stds}
\end{figure}

Given a configuration $x\in \alp^\Z$ and an interval $I=[i,j]$ with $i\leq j\in\Z$  denote $x[I]=x[i]x[i+1]\cdots x[j]\subseteq \alp^{j-i+1}$. For a CA $F:\alp^\Z\to \alp^\Z$, a configuration $x\in X$ and an interval $I=[i,j]$ with $i\leq j\in\Z$, the $I$-trace of $x$ (with respect to $F$) is the one-sided sequence $\tr_{F,I}(x)$ over the alphabet $\alp^{j-i+1}$ (i.e. the symbols of the alphabet are words over $\alp$) defined by $\tr_{F,I}(x)[t]=F^t(x)[I]$ for $t\in\N$. These correspond to columns of various width and position in the space-time diagram of $x$ with respect to $F$. If $I=\{i\}$ is the degenerate interval, we may write $\tr_{F,i}(x)$ and if $i=0$, we may write $\tr_F(x)$. If the CA $F$ is clear from the context, we may write $\tr_I(x)$.

\section{Rapidly Left Expansive Cellular Automata}\label{rapidSect}

In this section we define the class of rapidly left expansive cellular automata. This definition is strongly guided by the proof of Theorem~\ref{aperThm}, which we also give in this section. One main component of the definition is the notion of left expansivity. This is a special case of the notions of expansive and one-sided expansive directions~\cite{BL97,CK15} for more general dynamical systems, and has appeared earlier in the context of cellular automata e.g. in~\cite{JK18}.

\begin{definition}
A set of the form $R(h,d,w)=\{(i,j)\in\Z^2\mid -d\leq i\leq h, 0\leq j< w\}$ with $h,d\in\N, w\in\Zpos$ is called a \emph{rectangle} (the rectangle of height $h$, depth $d$ and width $w$, or the rectangle of dimensions $(h,d,w)$). A CA $F$ is \emph{left expansive} (with dimensions $(h,d,w)$) if for $R=R(h,d,w)$, for any pair of space-time diagrams $\theta_1,\theta_2$ with respect to $F$ and any pair of points $(i_1,j_1),(i_2,j_2)\in\Z^2$ satisfying $(i_1,j_1)+R,(i_2,j_2)+R\in\Z\times(-\N)$ the implication
\[\theta_1|_{(i_1,j_1)+R} = \theta_2|_{(i_2,j_2)+R}\implies\theta_1[(i_1-1,j_1)] = \theta_2[(i_2-1,j_2)]\]
holds, see the left hand side of Figure~\ref{rectangle}. 
\end{definition}

\begin{figure}[ht]
\centering
\begin{tikzpicture}[scale=0.6]

\draw[help lines] (0, -4) grid (9, 3);
\draw[help lines] (10, -4) grid (16, 3);

\def\nhbr[#1,#2](#3,#4){\draw (#3,#4) -- ++(-#1,0) -- ++(0,1) -- ++(#1+#2+1,0) -- ++(0,-1) -- ++(-#2,0); \draw (#3,#4) -- ++(0,-1) -- ++(1,0) -- ++(0,1)}

\def\nhbrf[#1,#2](#3,#4){\fill[color=lightgray] (#3,#4) -- ++(0,-1) -- ++(1,0) -- ++(0,1); \nhbr[#1,#2](#3,#4)}

\def\nhbrfperm[#1,#2](#3,#4){\draw[pattern=north east lines] (#3,#4) ++(-#1+1,0) rectangle ++(#1+#2,1); \fill[color=lightgray] (#3,#4) ++ (-#1,0) -- ++(0,1) -- ++(1,0) -- ++(0,-1); \nhbrf[#1,#2](#3,#4)}

\def\exp[#1,#2,#3](#4,#5){\fill[color=lightgray] (#4,#5) ++ (-1,0) -- ++(0,1) -- ++(1,0) -- ++(0,-1); \draw[thick,dashed] (#4,#5) -- ++(-1,0) -- ++(0,1) -- ++(1,0); \draw[thick] (#4,#5) ++ (0,1) -- ++(0,#1) -- ++(#3+1,0) -- ++(0,-#1-#2-1) -- ++(-#3-1,0) -- ++(0,#2+1);}

\exp[1,2,1](2,0);
\draw [decorate,decoration={brace,amplitude=5pt}] (2,1) -- ++(0,1) node[midway,xshift=-10pt]{$h$};
\draw [decorate,decoration={brace,mirror,amplitude=5pt}] (2,0) -- ++(0,-2) node[midway,xshift=-10pt]{$d$};
\draw [decorate,decoration={brace,amplitude=5pt}] (2,2) -- ++(2,0) node[midway,yshift=10pt]{$w$};

\exp[1,2,1](6,-1);

\nhbrfperm[2,1](13,1);
\draw [decorate,decoration={brace,amplitude=5pt}] (11,2) -- ++(2,0) node[midway,yshift=10pt]{$m$};
\draw [decorate,decoration={brace,amplitude=5pt}] (14,2) -- ++(1,0) node[midway,yshift=10pt]{$n$};

\nhbrfperm[2,1](13,-2);
\exp[0,1,2](12,-2);
\draw [decorate,decoration={brace,amplitude=5pt}] (12,-1) -- ++(3,0) node[midway,yshift=10pt]{$w=m+n$};
\draw [decorate,decoration={brace,mirror,amplitude=5pt}] (12,-2) -- ++(0,-1) node[midway,xshift=-18pt]{$d=1$};

\node at (11.5,1.5) {$a$};
\node[fill=white] at (13.5,1.5) {$v$};
\node at (13.5,0.5) {$b$};

\end{tikzpicture}
\caption{Left: A pair of translations of $R(h,d,w)$ enclosed by thick lines. Assuming these translated rectangles have identical contents in a space-time diagram (or even if they are in different space-time diagrams) of a left expansive CA with dimensions $(h,d,w)$, then also the symbols contained in the gray squares are identical.\\
Right: An $(m,n)$ left permutive CA is left expansive. }
\label{rectangle}
\end{figure}
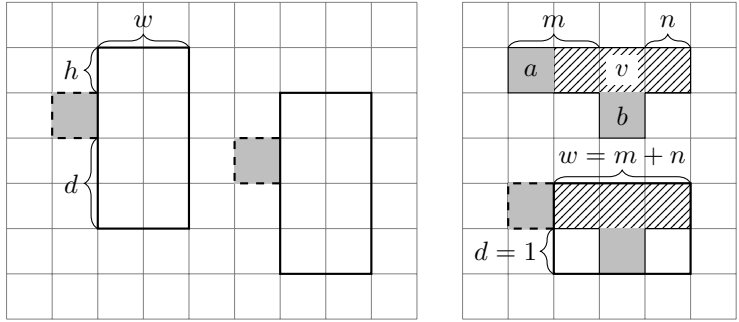

Given a CA $F$ with an $(m,n)$ local rule $f$ (upper right of Figure~\ref{rectangle}) the contents of the left gray cell ($a\in\alp$) and the striped area ($v\in\alp^{m+n}$) in a space-time diagram determine the content of the bottom gray cell via $f$ by $b=f(av)$. If $F$ is additionally $(m,n)$ left permutive, the content of the left gray cell can be expressed in terms of the content of the striped area and the bottom gray cell as $a=f_v^{-1}(b)$. In particular (lower right of Figure~\ref{rectangle}), the contents of the area enclosed by thick lines determines the content of the gray cell enclosed by dashed lines, and therefore an $(m,n)$ left permutive CA is left expansive with dimensions $(0,1,m+n)$. A fractional multiplication automaton is left expansive with dimensions $(1,1,1)$ by Proposition~3.7 of~\cite{Kop21trace}. 

Intuitively left expansivity of $F:\alp^\Z\to\alp^\Z$ with dimensions $(h,d,w)$ means that there is a CA with an $(h,d)$ local rule over the alphabet $\alp^w$ that treats the columns of space-time diagrams of $F$ as configurations. The following lemma demonstrates one way to make use of left expansivity: if a sufficiently wide column in a space-time diagram in eventually periodic, so are also columns to the left of it. This corresponds to the fact that the image of an eventually periodic configuration via a CA is also eventually periodic.

\begin{lemma}\label{expPer}
If $F:\alp^\Z\to\alp^\Z$ is left expansive with dimensions $(h,d,w)$ and $x\in\alp^\Z$, $i\in\Z$ are such that $\tr_{[i,i+w-1]}(x)$ is eventually $p$-periodic with preperiod $c$, then $\tr_{[i-1,(i+w-1)-1]}(x)$ is eventually $p$-periodic with preperiod $c+h$.
\end{lemma}
\begin{proof}
Let $R=R(h,d,w)$ and let $\theta$ be the space-time diagram of $x$ with respect to $F$. Let $t\geq c+h$. By the assumption of eventual periodicity, for all $(r_1,r_2)\in R$ it holds that
\begin{flalign*}
\theta[(i,-t)+(r_1,r_2)]&=\tr_{i+r_1}(x)[t+r_2] \\
&=\tr_{i+r_1}(x)[t+p+r_2]=\theta[(i,-(t+p))+(r_1,r_2)].
\end{flalign*}
Therefore $\theta|_{(i,-t)+R}=\theta|_{(i,-(t+p))+R}$ and it follows that 
\[\tr_{i-1}(x)[t]=\theta[(i-1,-t)]=\theta[(i-1,-(t+p)]=\tr_{i-1}(x)[t+p].\]
Combining this with $\tr_{[i,i+w-1]}(x)[t]=\tr_{[i,i+w-1]}(x)[t+p]$ yields
\[\tr_{[i-1,(i+w-1)-1]}(x)[t]=\tr_{[i-1,(i+w-1)-1]}(x)[t+p].\]
Since $t\geq c+h$ is arbitrary, $\tr_{[i-1,(i+w-1)-1]}(x)$ is eventually $p$-periodic with preperiod $c+h$.
\end{proof}

Another main component for the definition of rapidly left expansive CA is the notion of a spreading speed. This is similar to a Lyapunov exponent~\cite{She92,Tis00}, but depends only on the action of the CA on the elements of $\num(\alp)$.

\begin{definition}
A CA $F:\alp^\Z\to \alp^\Z$ is \emph{left spreading} (on number-like configurations) if $F(0^\Z)=0^\Z$ and for each $x\in\num(\alp)$ there is a $t\in\Zpos$ such that $\lbound(F^t(x))<\lbound(x)$. The \emph{spreading speed} of a left spreading $F$ is
\[\sup_{x\in\num(\alp)}\limsup_{t\in\Zpos}\frac{\lbound(x)-\lbound(F^t(x))}{t}=\sup_{x\in\num(\alp)}\limsup_{t\in\Zpos}\frac{-\lbound(F^t(x))}{t}.\]
\end{definition}

Multiplying a positive real number by a fraction $p/q$ causes the base $pq$ representation to lengthen by approximately $t\log_{pq}(p/q)$ symbols to the left. Due to this fractional multiplication automata are left spreading with spreading speed $\log_{pq}(p/q)$. An elementary CA is left spreading if and only if its local rule maps the triplet $001$ to $1$, and then its spreading speed is equal to $1$.

We are now ready to present the main definition of this paper.

\begin{definition}
A CA $F:\alp^\Z\to \alp^\Z$ is \emph{rapidly left expansive} (with width $w\in\Zpos$) if
\begin{itemize}
\item $F$ is left expansive with dimensions $(h,d,w)$,
\item $F$ is left spreading with spreading speed $s$,
\item the inequality $s<1/h$ is satisfied.
\end{itemize}
\end{definition}
By the above discussion, this class of automata contains in particular all the fractional multiplication automata (because then $1/h=1>\log_{pq}(p/q)=s$) and all left permutive left spreading CA (because $1/0=\infty>s$) such as Rule~30.

The following theorem has been proved earlier in Proposition~3 of~\cite{Jen90} for left permutive left spreading elementary CA and in Proposition~3.8 of~\cite{Kop21trace} for fractional multiplication automata. We reprove it (with essentially the same proof) for general rapidly left expansive CA.

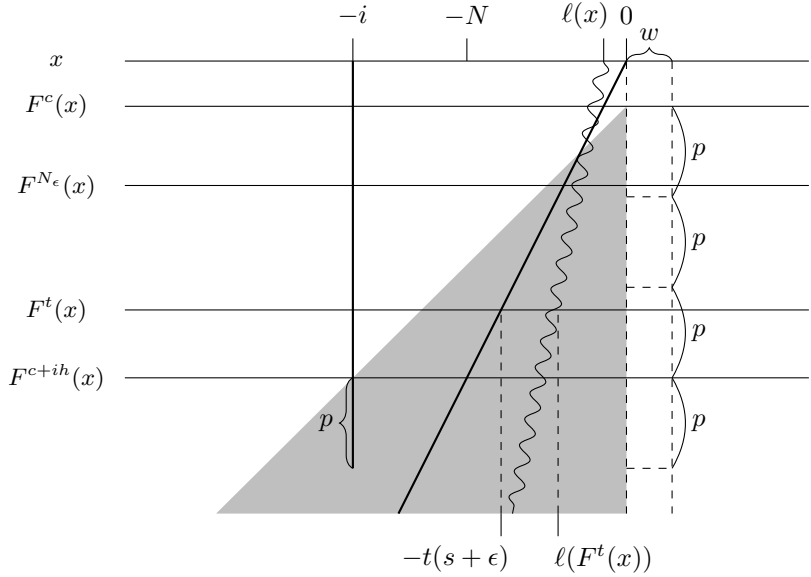
\begin{figure}[ht]
\centering
\begin{tikzpicture}[scale=0.6]

\def\exp[#1,#2,#3](#4,#5){\fill[color=lightgray] (#4,#5) ++ (-1,0) -- ++(0,1) -- ++(1,0) -- ++(0,-1); \draw[thick,dashed] (#4,#5) -- ++(-1,0) -- ++(0,1) -- ++(1,0); \draw[thick] (#4,#5) ++ (0,1) -- ++(0,#1) -- ++(#3+1,0) -- ++(0,-#1-#2-1) -- ++(-#3-1,0) -- ++(0,#2+1);}

\fill[color=lightgray] (11,9) -- ++(-9,-9) -- ++(9,0);
\draw[thick] (11,10) -- ++(-5,-10);
\draw decorate[decoration=snake] {(10.5,10) -- (8.5,0)};

\node at (-1.5,10) {\small $x$}; \draw (0,10) -- ++(15,0);
\node at (-1.5,9) {\small $F^c(x)$}; \draw (0,9) -- ++(15,0);
\node at (-1.5,7.25) {\small $F^{N_\epsilon}(x)$}; \draw (0,7.25) -- ++(15,0);
\node at (-1.5,4.5) {\small $F^{t}(x)$}; \draw (0,4.5) -- ++(15,0);
\node at (-1.5,3) {\small $F^{c+ih}(x)$}; \draw (0,3) -- ++(15,0);

\draw[dashed] (8.25,0) -- ++(0,4.5);
\draw (8.25,0) -- ++(0,-0.5) ++(-1,-0.5) node{$-t(s+\epsilon)$};
\draw[dashed] (9.5,0) -- ++(0,4.5);
\draw (9.5,0) -- ++(0,-0.5) ++(1,-0.5) node{$\lbound(F^t(x))$};

\draw[thick] (5,10) -- ++(0,-9);
\draw (5,10) -- ++(0,0.5) ++(0,0.5) node{$-i$};
\draw [decorate,decoration={brace,mirror,amplitude=5pt}] (5,3) -- ++(0,-2) node[midway,xshift=-10pt]{$p$};
\draw (7.5,10) -- ++(0,0.5) ++(0,0.5) node{$-N$};

\draw (10.5,10) -- ++(0,0.5) ++(-0.4,0.5) node{$\lbound(x)$};
\draw[dashed] (11,10) -- ++(0,-10);
\draw (11,10) -- ++(0,0.5) ++(0,0.5) node{$0$};
\draw[dashed] (12,10) -- ++(0,-10);
\draw [decorate,decoration={brace,amplitude=5pt}] (11,10) -- ++(1,0) node[midway,yshift=10pt]{$w$};

\draw (12,9) to[out=-60, in=60] node[midway,xshift=5pt]{$p$} ++(0,-2) ;
\draw (12,7) to[out=-60, in=60] node[midway,xshift=5pt]{$p$} ++(0,-2) ;
\draw[dashed] (11,7) -- ++(1,0);
\draw (12,5) to[out=-60, in=60] node[midway,xshift=5pt]{$p$} ++(0,-2) ;
\draw[dashed] (11,5) -- ++(1,0);
\draw (12,3) to[out=-60, in=60] node[midway,xshift=5pt]{$p$} ++(0,-2) ;
\draw[dashed] (11,1) -- ++(1,0);

\end{tikzpicture}
\caption{The proof of Theorem~\ref{aperThm}. By left expansivity, all columns within the gray area are $p$-periodic. Find an $N\in\N$ such that, at all coordinates $-i$ to the left of $-N$, you can extend a line (the thick vertical line in the diagram) down to depth $p$ within the gray area but which is fully above the wavy line following the left bound $\lbound(F^t(x))$. To see that this is possible, construct a line (the thick line of slope $1/(s+\epsilon)$ in the diagram) that eventually stays above the wavy line but whose slope is greater than the slope $h$ of the boundary of the gray area.}
\label{aperTrace}
\end{figure}

\begin{theorem}\label{aperThm}
If $F:\alp^\Z\to \alp^\Z$ is rapidly left expansive with width $w\in\Zpos$ and $x\in\num(\alp)$, then $\tr_{[i,i+w-1]}(x)$ is not eventually periodic for any $i\in\Z$.
\end{theorem}
\begin{proof}
Assume that $F$ is left expansive with dimensions $(h,d,w)$ and left spreading with speed $s$ so that $s<1/h$. Assume to the contrary and without loss of generality (by replacing $x$ with $\sigma^i(x)$ if necessary) that $\tr{[0,w-1]}(x)$ is eventually $p$-periodic with preperiod $c\in\N$. A simple induction based on Lemma~\ref{expPer} shows that $\tr_{[-i,-i+w-1]}(x)$ (and in particular $\tr_{-i}(x)$) is eventually $p$-periodic with preperiod $c+ih$ for all $i\in\N$. By periodicity we have for all $i\in\N$ the implication
\[\tr_{-i}(x)[t]=0\mbox{ for }0\leq t<c+ih+p \implies \tr_{-i}(x)=0^\N.\]
Because of this, if there exists an $N\in\N$ such that the left hand side of this implication is satisfied for all $i\geq N$, then $F^t(x)[-i]=0$ whenever $t,i\in\N$, $i\geq N$, contradicting the left spreading property. We outline in Figure~\ref{aperTrace} a visual proof for the existence of such an $N$ and present more details below. 

For an $N\in\N$ as in the previous paragraph to exist it is sufficient for all sufficiently large $i\in\N$ to satisfy $-i<\lbound(F^t(x))$ whenever $0\leq t<c+ih+p$. By the left spreading property
\[s\geq\limsup_{t\in\Zpos}\frac{-\lbound(F^t(x))}{t}.\]
Thus for any $\epsilon>0$ there is an $N_{\epsilon}\in\N$ such that for all $t\geq N_{\epsilon}$ it holds that
\[s+\epsilon\geq\frac{-\lbound(F^t(x))}{t}\]
and $-t(s+\epsilon)\leq\lbound(F^t(x))$. Since by assumption $s<1/h$, we may fix $\epsilon$ so that $s+\epsilon<1/h$. To conclude it is sufficient to show that all sufficiently large $i\in\N$ satisfy $-i<-t(s+\epsilon)$ for $N_{\epsilon}\leq t<c+ih+p$ and $-i<\lbound(F^t(x))$ for $0\leq t<N_\epsilon$. The latter condition is satisfied as long as $i$ is sufficiently large and the former condition is satisfied when $(c+ih+p)(s+\epsilon)<i$. If $h=0$, this is satisfied for all sufficiently large $i$. If $h>0$, dividing both sides by $ih$ results in
\[\frac{(c+p)(s+\epsilon)}{ih}+(s+\epsilon)<\frac{1}{h}.\]
Letting $i$ tend to infinity, we see that this inequality is satisfied for all sufficiently large $i$ if $s+\epsilon<1/h$. This holds by our choice of $\epsilon$. 
\end{proof}

\begin{remark}
The shift map $\sigma:\alp^\Z\to\alp^\Z$ is left expansive with dimensions $(h,d,w)=(1,0,1)$ and left spreading with spreading speed $s=1$, but from $s=1/h$ it follows that $\sigma$ is not rapidly left expansive. The map $\sigma$ does not satisfy the conclusion of Theorem~\ref{aperThm}, because $\tr(x)$ is eventually periodic whenever $x\in\num(\alp)$ is eventually periodic.
\end{remark}

\section{Results}\label{resSect}

In this section we present two new results on rapidly left expansive cellular automata. The proofs of both of these results utilize Theorem~\ref{aperThm} as their final step. The fractional part of a number $\xi\in\R$ is
\[\fractional(\xi)=\xi-\lfloor \xi\rfloor\in[0,1).\]
By abuse of notation, for $x\in \alp^\Z$ and $c\in\Z$ we define $\fractional_c(x)\in \alp^{\N}$ by $\fractional_c(x)[i]=x[c+i]$ for all $i\in\N$ (in the special case $c=0$ the subscript $c$ may be dropped).

Whenever $\xi>0$ and $p>q>1$, it is known that $\fractional(\xi)$ can occur in the sequence $(\fractional(\xi(p/q)^i))_{i\in\N}$ only finitely many times by Lemma~2.1 of~\cite{Dub08}. Without loss of generality $p$ and $q$ are coprime, and then this is equivalent to the statement that any $x\in\num(\digs_{pq})$ can appear in the sequence $(\fractional(\Mul_{p/q,pq}^i(x))_{i\in\N}$ only finitely many times. We will now show that this result generalizes to the case where $\Mul_{p/q,pq}$ is replaced by an arbitrary rapidly left expansive cellular automaton. We first recall the Morse-Hedlund theorem.

\begin{theorem}[Morse and Hedlund, \cite{MH38}, Theorem~7.4]\label{MH}
If $x\in\alp^\N$ is not eventually periodic, then at least $n+1$ distinct words of length $n$ occur in $x$ for each $n\in\Zpos$.
\end{theorem}

\begin{lemma}\label{mulRepPer}
Let $F:\alp^\Z\to \alp^\Z$ be left expansive with dimensions $(h,d,w)$ and let $x\in \alp^\Z$. If $t\in\Zpos$, $N\in\N$ and $K=\abs{\alp}^{tw}$ satisfy $(h+d)K\leq tN$ and $\fractional_c(F^{it}(x))=\fractional_c(x)$ for some $c\in\Z$ and all $0\leq i\leq N$, then $x$ is eventually periodic.
\end{lemma}
\begin{proof}
We may assume without loss of generality (by replacing $x$ with $\sigma^c(x)$ if necessary) that $c=0$. Assume that $F$ is a radius $r$ CA and define
\begin{flalign*}
& T_n=\{\tr_{[i,i+w]}(x)[0,n-1]\mid i\geq r(t-1)\} \quad \mbox{for } n\in\N, \\
& U=\{F^{hK}(x)[i-K,i-1]\mid i\geq r(t-1)\}.
\end{flalign*}
Because $\fractional(F^{it}(x))=\fractional(x)$ holds for $0\leq i\leq N$, it follows that the word $\tr_{[i,i+w-1]}(x)[0,tN-1]$ is a concatenation of $N$ copies of $\tr_{[i,i+w-1]}(x)[0,t-1]$ when $i\geq r(t-1)$. This together with the inequality $(h+d)K\leq tN$ implies that
\[K\geq\abs{T_t}=\abs{T_{tN}}\geq\abs{T_{(h+d)K}}.\]
By left expansivity the surjective mapping $T_{(h+d)K}\to U$ defined by
\[\tr_{[i,i+w]}(x)[0,(h+d)K-1]\mapsto F^{hK}(x)[i-K,i-1]\]
is well defined, so $U$ contains at most $K$ words of length $K$. By Theorem~\ref{MH} the sequence $F^{hK}(x)$ is eventually periodic. Sufficiently many applications of $F$ transform $F^{hK}(x)$ to $F^{tN}(x)$ and $\fractional(F^{tN}(x))=\fractional(x)$, so $x$ is also eventually periodic.
\end{proof}

It turns out that the assumption of $\fractional(x)$ repeating many times periodically in the sequence $(\fractional(F^{i}(x))_{i\in\N}$ can be significantly weakened.

\begin{lemma}\label{twoRepPer}
If $F:\alp^\Z\to \alp^\Z$ is left expansive, $x\in \alp^\Z$, and there exists a $t\in\Zpos$ such that $\fractional(F^t(x))=\fractional(x)$, then $x$ is eventually periodic.
\end{lemma}
\begin{proof}
This claim follows from Lemma~\ref{mulRepPer} if we can show that for every $N\in\Zpos$ there is a $c\in\N$ such that $\fractional_c(F^{it}(x))=\fractional_c(x)$ for all $0\leq i\leq N$. We show this by induction, so let $N\in\Zpos$ and $c\in\N$ be such that $\fractional_c(F^{it}(x))=\fractional_c(x)$ holds for all $0\leq i\leq N$. Assuming that $F^t$ has radius $r$, applying $F^t$ to the configurations $F^{it}(x)$ yields $\fractional_{c+r}(F^{(i+1)t}(x))=\fractional_{c+r}(F^t(x))$ for all $0\leq i\leq N$, or equivalently $\fractional_{c+r}(F^{it}(x))=\fractional_{c+r}(F^t(x))$ for all $1\leq i\leq N+1$. It remains to show that $\fractional_{c+r}(F^t(x))=\fractional_{c+r}(x)$, but this follows from the assumption $\fractional(F^t(x))=\fractional(x)$.
\end{proof}

\begin{figure}[ht]
\centering
\begin{tikzpicture}[scale=0.4]

\foreach \i in {-1,...,8} {\fill[color=lightgray] (2+2*\i,11) -- (2+2*\i,10-\i) -- ++(0,-1) -- ++(2,0) -- (2+2*\i+2,11);}
\fill[color=lightgray] (20,11) -- ++(0,-10) -- ++(6,0) -- ++(0,10);
\foreach \i in {-1,...,2} {\fill[color=lightgray] (2+2*\i,1) -- (2+2*\i,-\i) -- ++(0,-1) -- ++(2,0) -- (2+2*\i+2,1);}
\fill[color=lightgray] (8,1) -- ++(0,-4) -- ++(18,0) -- ++(0,4);

\draw[help lines] (0, -3) grid (26,11);
\node at (-1,10.5) {\small $x$}; \draw (0,10) -- ++(26,0); \draw (0,11) -- ++(26,0);
\node at (-1,2.5) {\small $F^j(x)$};
\node at (-1,0.5) {\small $F^t(x)$}; \draw (0,0) -- ++(26,0); \draw (0,1) -- ++(26,0);

\draw[dashed] (18,-3) -- ++(0,14);
\draw[dashed] (10,-3) -- ++(0,14);
\draw[dashed] (4,-3) -- ++(0,14);

\def\exp[#1,#2,#3](#4,#5){\fill[color=gray] (#4,#5) ++ (-1,0) -- ++(0,1) -- ++(1,0) -- ++(0,-1); \draw[thick,dashed] (#4,#5) -- ++(-1,0) -- ++(0,1) -- ++(1,0); \draw[thick] (#4,#5) ++ (0,1) -- ++(0,#1) -- ++(#3+1,0) -- ++(0,-#1-#2-1) -- ++(-#3-1,0) -- ++(0,#2+1);}

\draw [decorate,decoration={brace,mirror,amplitude=5pt}] (10,-1) -- ++(4,0) node[midway,yshift=-10pt]{$p$};

\draw [decorate,decoration={brace,mirror,amplitude=5pt}] (0,10) -- ++(2,0) node[midway,yshift=-10pt]{$m$};

\draw (0.5,-3) -- ++(0,-0.5) ++(0,-0.5) node{$0$};
\draw (4.5,-3) -- ++(0,-0.5) ++(0,-0.5) node{$c$};
\draw (9.5,-3) -- ++(0,-0.5) ++(0,-0.5) node{$c'$};
\draw (18.5,-3) -- ++(0,-0.5) ++(0,-0.5) node{$(t-1)m$};

\exp[3,3,1](10,2);
\exp[3,3,1](14,2);

\end{tikzpicture}
\caption{A space-time diagram for the proof of Lemma~\ref{boundedPrePer}. Initially $p$-periodicity is guaranteed within the light gray area. To show $p$-periodicity up to the horizontal coordinate $c$, slide two rectangles around the diagram so that they remain at horizontal distance $p$ from each other. Identical contents within the rectangles imply identical contents within the dark gray squares. The starting point of the induction is at the horizontal coordinate $(t-1)m$ and proceeds leftwards.}
\label{pperBound}
\end{figure}
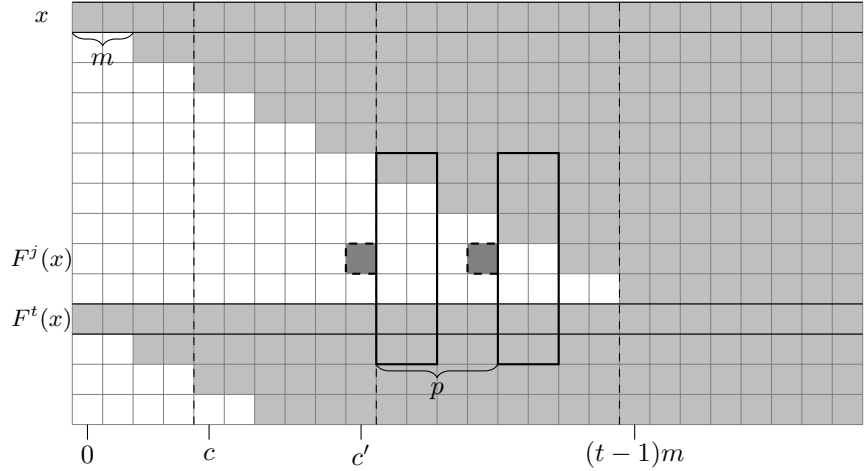

\begin{lemma}\label{boundedPrePer}
Assume that $F:\alp^\Z\to \alp^\Z$ with memory $m$ is left expansive with dimensions $(e,e,w)$ (height and depth are equal to $e$) and let $c=m(e-1)$. Then for any $x\in \alp^\Z$ and $t\in\N$ with $\fractional(x)$ and $\fractional(F^t(x))$ $p$-periodic, the sequence $\fractional_c(F^i(x))$ is $p$-periodic for all $0\leq i\leq t$.
\end{lemma}
\begin{proof}
Since $\fractional(x)$ and $\fractional(F^t(x))$ are $p$-periodic and $F$ has memory $m$, the sequences $\fractional(F^i(x))$ and $\fractional(F^{t+i}(x))$ are eventually $p$-periodic with preperiod $im$ for $0\leq i<t$. In particular, for $0\leq i<e$, they are $p$-periodic with preperiod $c$. Here we can argue using left expansivity as in Figure~\ref{pperBound}.

More precisely, to prove the lemma assume to the contrary that there are $c'\geq c$ and $e\leq j< t$ such that $\fractional_{c'}(F^j(x))$ is not $p$-periodic. Furthermore, assume that the choice of $c'$ is maximal, meaning that $\fractional_{c'+1}(F^i(x))$ is $p$-periodic for $0\leq i<t+e$. Let $R=R(e,e,w)$ and let $\theta$ be the space-time diagram of $x$ with respect to $F$. Then $\theta[(c'+1,-j)+R]=\theta[(p,0)+(c'+1,-j)+R]$ and therefore $F^j(x)[c']=\theta[(c',-j)]=\theta[(c'+p,-j)]=F^j(x)[c'+p]$. Thus $\fractional_{c'}(F^j(x))$ is $p$-periodic, a contradiction.
\end{proof}

\begin{theorem}
If $F:\alp^\Z\to \alp^\Z$ is rapidly left expansive and $x\in\num(\alp)$, then $\fractional(F^t(x))=\fractional(x)$ for finitely many $t\in\N$.
\end{theorem}
\begin{proof}
Assume to the contrary that $\fractional(F^t(x))=\fractional(x)$ for infinitely many $t\in\N$. In particular there is a positive such $t$ and thus by Lemma~\ref{twoRepPer} we may assume (after replacing $x$ with $\sigma^i(x)$ for a suitable $i\in\N$ if necessary) that $\fractional(x)$ is periodic. Since there are infinitely many such $t$, by Lemma~\ref{boundedPrePer} there is a $c\in\N$ such that $\fractional_c(F^i(x))$ is periodic for all $i\in\N$.

For each $i\in\N$ let $x_i$ be the unique periodic configuration that satisfies $\fractional_c(x_i)=\fractional_c(F^i(x))$. Clearly $x_{i+1}=F(x_i)$ for each $i\in\N$, from which it follows that $F^t(x_0)=x_t=x_0$. In particular $\tr_{[c,c+n]}(x)=\tr_{[c,c+n]}(x_0)$ is periodic for arbitrarily large $n\in\N$, contradicting Theorem~\ref{aperThm}.
\end{proof}

We proceed to the second main result. It is a special case of Theorem~2 in~\cite{Pis46} that, whenever $\xi>0$ and $p>q>1$, the sequence $(\fractional(\xi(p/q)^i))_{i\in\N}$ has infinitely many limit points in the interval $[0,1]$. Again without loss of generality $p$ and $q$ are coprime, and then this is equivalent to the statement that for any $x\in\num(\digs_{pq})$ the sequence $(\fractional(\Mul_{p/q,pq}^i(x))_{i\in\N}$ has infinitely many limit points in $\digs_{pq}^\N$. This result also generalizes to the case where $\Mul_{p/q,pq}$ is replaced by any rapidly left expansive cellular automaton.  

\begin{theorem}
Let $F:\alp^\Z\to \alp^\Z$ be a CA and let $x\in \alp^\Z$ be such that the sequence $\{\fractional(F^t(x))\}_{t\in\N}$ has finitely many limit points in $\alp^\N$. Then for every $w\in\N$ there is an $i\in\N$ such that $\tr_{[i,i+w]}(x)$ is eventually periodic.
\end{theorem}
\begin{proof}
Assume to the contrary that there are $k\in\Zpos$ different limit points and assume that $F$ has radius $r\in\N$. For $0\leq i\leq j$ let
\[W_{i,j}=\{v\in \alp^{j-i+1}\mid v=F^t(x)[i,j]\mbox{ for infinitely many }t\in\N\},\quad f(i,j)=\abs{W_{i,j}}.\]
Clearly $1\leq f(i,j)\leq k$. This quantity is monotonous in the sense that if $i'\leq i$ and $j\leq j'$, then $f(i,j)\leq f(i',j')$. It is possible to fix $i\leq kr-r$ and $j\geq kr+w+r$ so that $f(i,j)=f(i+r,j-r)$, because otherwise
\[f(0,2kr+w)>f(r,(2k-1)r+w)>f(2r,(2k-2)r+w)>\cdots>f(kr,kr+w)\geq 1\]
and $f(0,kr+w+kr)\geq k+1$.

We may assume without loss of generality (by replacing $x$ with $F^t(x)$ for a sufficiently large $t\in\N$ if necessary) that $F^t(x)[i,j]\in W_{i,j}$ for all $t\in\N$.  Let $p\in\Zpos$ be such that $F^p(x)[i,j]=x[i,j]$. We will prove by induction that $\tr_{[i,j]}(x)$ is $p$-periodic, i.e. $F^{p+t}(x)[i,j]=F^t(x)[i,j]$ for all $t\in\N$. The base case $t=0$ follows by the choice of $p$, so let $t\in\N$ be such that $F^{p+t}(x)[i,j]=F^t(x)[i,j]$. Since $F$ has radius $r$, it follows that $F^{p+t+1}(x)[i+r,j-r]=F^{t+1}(x)[i+1,j-r]$. Since $f(i,j)=f(i+r,j-r)$, it follows that $F^{p+t+1}(x)[i,j]=F^{t+1}(x)[i,j]$, which proves the induction step. Since $\tr_{[i,j]}(x)$ is $p$-periodic and $j-i\geq w$, also $\tr_{[i,i+w]}(x)$ is $p$-periodic.
\end{proof}

\begin{theorem}
If $F:\alp^\Z\to \alp^\Z$ is rapidly left expansive and $x\in\num(\alp)$, then the sequence $\{\fractional(F^t(x))\}_{t\in\N}$ has infinitely many limit points in $\alp^\N$.
\end{theorem}
\begin{proof}
Assume to the contrary that $\{\fractional(F^t(x))\}_{t\in\N}$ has finitely many limit points. By the previous theorem for any $w\in\N$ there is an $i\in\N$ such that $\tr_{[i,i+w]}(x)$ is eventually periodic. For a sufficiently large $w$ this contradicts Theorem~\ref{aperThm}. 
\end{proof}

We conclude by noting that perhaps the most famous question~\cite{Wol19} concerning Rule~30 remains unsolved. It concerns the trace of width $1$ of a single, very simple configuration, but it is probably equally difficult for all configurations of $\num(\digs_2)$. Note that the case of trace of width 2 is covered by Theorem~\ref{aperThm}. The answer ``no'' is expected.
\begin{problem}
Let $x=\cdots0001000\cdots\in\digs_2^\Z$ be the configuration containing a single $1$ at the origin. Is $\tr_{W_{30},0}(x)$ eventually periodic?
\end{problem}

Ideally one could even solve this problem for some natural class of CA that contains Rule 30 and fractional multiplication automata (although Theorem~\ref{aperThm} already covers the latter). Unfortunately the set of rapidly left expansive CA cannot be such a class, because it contains the additive ECA Rule~90 (with a $(1,1)$ local rule $f(abc)=a+c\pmod{2}$) that produces a single eventually periodic column starting from the configuration with a single $1$ at the origin.

\section*{Acknowledgements}
The work was supported by the Finnish Cultural Foundation.

\bibliographystyle{plain}
\bibliography{mybib}{}

\end{document}